\documentclass[a4paper,12pt]{amsart}

\usepackage[english]{babel}
\usepackage{amsthm, latexsym, gastex, amssymb,url,listings,enumitem}
\usepackage[mathcal]{eucal}
\usepackage{graphicx,url}

\DeclareMathOperator{\lcm}{lcm}

\renewcommand{\phi}[0]{\varphi}
\renewcommand{\theta}[0]{\vartheta}
\renewcommand{\epsilon}[0]{\varepsilon}

\newcommand{\N}{\text{$\mathbb{N}$}}
\newcommand{\Z}{\text{$\mathbb{Z}$}}

\newcommand{\Pro}{\text{$\mathbf{P}^1$}}
\newcommand{\F}{\text{$\mathbb{F}$}}
\newcommand{\Fm}{\text{$\mathbb{F}_{2^m}$}}
\newcommand{\Fn}{\text{$\mathbb{F}_{2^n}$}}

\newcommand{\Fq}{\text{$\mathbb{F}_q$}}
\newcommand{\Fqt}{\text{$\mathbb{F}_{q^2}$}}
\newcommand{\cFn}{\text{$\overline{\mathbb{F}_{2^n}}$}}
\newcommand{\cFq}{\text{$\overline{\mathbb{F}_{q}}$}}

\newtheorem{theorem}{Theorem}[section]
\newtheorem{lemma}[theorem]{Lemma}
\newtheorem{proposition}[theorem]{Proposition}

\theoremstyle{definition}

\newtheorem{example}[theorem]{Example}

\theoremstyle{remark}
\newtheorem{remark}[theorem]{Remark}

\numberwithin{equation}{section}

\begin{document}

\bibliographystyle{amsplain}

\date{}

\title[On the maps $ax^{2^k}+b$ and $(a x^{2^k} + b)^{-1}$]
{On the iterations of  the maps $ax^{2^k}+b$ and $(a x^{2^k} + b)^{-1}$ over finite fields of characteristic two}

\author{S.~Ugolini}
\email{s.ugolini@unitn.it} 



\begin{abstract}  
The maps $x \mapsto ax^{2^k}+b$ defined over finite fields of characteristic two can be related to the duplication map over binary supersingular elliptic curves. Relying upon the structure of the group of rational points of such curves we can describe the possible cycle lengths of the maps. Then we extend our investigation to the maps $x \mapsto (ax^{2^k}+b)^{-1}$. We also notice some relations between these latter maps and the polynomials $x^{2^k+1} + x +a$, which have been extensively studied in literature.  
\end{abstract}

\maketitle
\section{Introduction}\label{sec_intro}
In a paper \cite{blu} appeared in 2004,  Bluher studied the polynomials $x^{q+1} + a x + b$ defined over finite fields of characteristic $p$, with $q$ a power of $p$. The author studied the splitting field of such polynomials and other related problems, such as the number of roots in a given field. Later, other authors concentrated on such polynomials over finite fields of characteristic two.   

Let $P_a (x) := x^{2^k+1} + x + a$, where $k$ is a positive integer and $a \in \F^*_{2^n}$ for some positive integer $n$. In \cite{blu} the following result is proved.

\begin{theorem}
For any $a \in \F^*_{2^n}$ and a positive integer $k$, the polynomial $P_a(x)$ has either zero, one, two or $2^{\gcd(k,n)}+1$ zeros in $\Fn$.
\end{theorem}  

Helleseth and Kholosha \cite{hk1, hk2} proved that when $\gcd(n,k)=1$ the polynomial $P_a(x)$ has 0, 1 or 3 zeros in $\Fn$. Moreover, they provided criteria for $P_a (x)$ to have 0, 1 or 3 zeros depending upon the value of $a$.
In a paper appeared recently \cite{km} the authors calculated all possible roots of $P_a(x)$ in $\Fn$ and gave new criteria on $a$ such that the equation $P_a (x) = 0$ has $0, 1$ or $3$ solutions with the hypothesis that $\gcd(k,n)=1$.

Since $x = 1$ is not a root of $P_a(x)$, solving the equation $P_a(x) = 0$ amounts to finding all $x \in \Fn$ such that 
\[
x = \frac{a}{x^{2^k}+1}.
\]

From the point of view of dynamical systems, this latter problem is equivalent to finding all the fixed points in $\F_{2^n}$ of the map $x \mapsto \frac{a}{{x^{2^k}+1}}$ which can be written also as $x \mapsto \frac{1}{a^{-1} x^{2^k} + a^{-1}}$. Such a map is an instance of the maps $x \mapsto (ax^{2^k}+b)^{-1}$, which can be defined over any finite field $\Fn$ of characteristic two and then extended to the projective line $\Pro(\Fn)$. We study such maps in Section \ref{sec_x2k1}. In Section \ref{sec_ell} we begin our investigation studying the iterations of the map $x \mapsto ax^4+b$ over $\Pro(\Fn)$. Such maps are involved in the definition of the duplication map of certain supersingular elliptic curves defined over finite fields of characteristic two. In Section \ref{sec_x2k} we extend our investigation to more general maps $x \mapsto a x^{2^k} + b$ where $k$ is a non-negative integer.

\section{Preliminaries}\label{sec_pre}
In the paper we use the following notations.
\begin{itemize}
\item $\N$ denotes the set of non-negative integers and $\N^* := \N \backslash \{ 0 \}$.
\item If $a,b \in \N$ with $a \leq b$, then 
\begin{align*}
[a,b] & := \{n \in \N: a \leq n \leq b \};\\
]a, +\infty[ & := \{n \in \N : a < n \}.
\end{align*}
\item If $x, y, n \in \Z$, then $x \approx y \bmod n$ stands for 
\[
x \equiv y \bmod n \quad \text{ or } \quad x \equiv -y \bmod n.
\]
\item If $f : A \to A$ is a map defined on a set $A$ and $m \in \N^*$ then we denote by $f^m$ the $m$-fold composition of $f$ with itself.
\item If $\Fq$ is a finite field with $q$ elements, then $\Pro(\Fq) := \Fq \cup \{ \infty \}$. Moreover we denote by $\cFq$ the algebraic closure of $\Fq$.
\end{itemize}

If $f : \Pro(\Fq) \to \Pro(\Fq)$ is a map defined over $\Pro(\Fq)$, then we denote by $G_q (f)$ the digraph associated with $f$ over $\Pro(\Fq)$. The vertices of $G_q(f)$ are the elements of $\Pro(\Fq)$ and an arrow joins a vertex $\alpha$ to a vertex $\beta$ if $f(\alpha) = \beta$.
Similarly, if $f : \Pro(\cFq) \to \Pro(\cFq)$ then we denote by $G_{\overline{q}} (f)$ the digraph associated with $f$ over $\Pro(\cFq)$.

If $f : \Pro(\Fq) \to \Pro(\Fq)$ and $x_0 \in \Pro(\Fq)$, then we denote the orbit of $x_0$ by
\[
O(f, x_0) := \{f^i (x_0) : i \in \N \}.
\]
In case the map is not ambiguous, we use the shorter notation $O(x_0)$. If $f^j (x_0) = x_0$ for some $j \in \N^*$, then we say that $x_0$ is periodic with respect to $f$ or $f$-periodic (we also write that $x_0$ is periodic if the map $f$ is clear from the context).  If $l \in \N^*$ is the smallest of the indices $j \in \N^*$ such that $f^j (x_0) = x_0$, then we say that $l$ is the period of $x_0$ or the length of the orbit $O(x_0)$. We notice that the length of the orbit is equal to the cardinality of the orbit itself, namely
\[
l = |O(f,x_0)|.
\] We can prove that any $x_0 \in \Pro(\Fq)$ is periodic if $f$ is one-to-one.

\begin{lemma}\label{per_lemma}
Let $f : \Pro(\Fq) \to \Pro(\Fq)$.

\begin{enumerate}[leftmargin=*]
\item $f^i (x_0)$ is periodic for some $i \in \N$ if and only if $f^i (x_0) = f^k (x_0)$ for some $k \in ]i, + \infty[$.
\item If $f : \Pro(\Fq) \to \Pro(\Fq)$ is one-to-one and $x_0 \in \Pro(\Fq)$, then $x_0$ is periodic.
\end{enumerate} 
\end{lemma}
\begin{proof}
We prove separately the statements.
\begin{enumerate}[leftmargin=*]
\item If $f^i (x_0)$ is periodic then $f^j (f^i(x_0)) = f^i (x_0)$ for some $j \in \N^*$. Hence $f^i (x_0)= f^k (x_0)
$ with $k = i + j$. Vice versa, if $f^i (x_0) = f^k (x_0)$ for some $k \in ]i, + \infty[$ then $f^i (x_0)$ is periodic because $f^j (f^i(x_0)) = f^i (x_0)$ with $j := k-i$.
\item We consider the map $g : \N  \to \Pro(\Fq)$ defined as 
\[
g(n) = f^n (x_0)
\]
for any $n \in \N$.  Since $\Pro(\Fq)$ is a finite set, the image set $g(\N)$ is finite too. Hence there exist some $i, k \in \N$ with $i < k$ such that $f^i (x_0) = f^k (x_0)$. We define the set 
\[
I := \{i \in \N : f^i (x_0) \text{ is periodic} \}.
\]

Let $m := \min(I)$. We have that $f^m (x_0) = f^j (x_0)$ for some $j \in ]m, + \infty[$. Suppose that $m > 0$. Since
\[
f (f^{m-1} (x_0)) = f (f^{j-1} (x_0))
\]
and $f$ is one-to-one, we have that $f^{m-1} (x_0) = f^{j-1} (x_0)$, in contradiction with the minimality of $m$. Hence $m=0$, namely $x_0$ is periodic. \qedhere
\end{enumerate}
\end{proof}
\begin{remark}
Since for a one-to-one map $f : \Pro(\Fq) \to \Pro(\Fq)$ every $x_0 \in \Pro(\Fq)$ is periodic,  the graph $G_q (f)$ is formed by a finite number of cycles, possibly of different lengths.
\end{remark}

The lengths of the orbits of a map $f$ and its $m$-fold composition $f^m$ are related as described in the following lemma.

\begin{lemma}\label{lem_orbits}
Let $f$ be a map defined on $\Pro(\Fq)$ and  $g := f^m$ for some $m \in \N^*$. If $x_0 \in \Pro(\Fq)$ is $f$-periodic and
\[
l := |O(f, x_0)|,
\]  
then $x_0$ is $g$-periodic and
\[
|O(g,x_0)| = \frac{l}{\gcd(l,m)}.
\]
\begin{proof}
If $g^{h} (x_0) = x_0$ for some $h \in \N^*$, then $f^{hm} (x_0) = x_0$. Therefore $l$ divides $hm$. The result follows because $\frac{l}{\gcd(l,m)}$ divides $h$.
\end{proof} 
\end{lemma}

We recall the following results (see \cite[Proposition 7.1.2]{ir} and \cite[Theorem 1.68]{lid}).

\begin{proposition}\label{ros_lemma}
Let $\alpha \in \Fq^*$ and $n \in \N^*$. Then $x^n = \alpha$ has solutions if and only if $\alpha^{(q-1)/d} = 1$, where $d = \gcd(n, q-1)$. If there are solutions, then there are exactly $d$ solutions.
\end{proposition}

\begin{theorem}
Let $\F$ be a field. The element $b \in \F$ is a multiple root of $f \in \F [x]$ if and only if it is a root of both $f$ and $f'$.
\end{theorem}

Now we introduce some notations for the maps we investigate in this paper. 

Let $q := 2^n$ for some $n \in \N^*$.  If $a, b \in \Fq$ with $a \neq 0$ and $k \in \N^*$ then we define the maps $\theta_{a,b,k}$ and $\psi_{a,b,k}$ on $\Pro(\Fq)$ as follows:
\begin{align*}
\theta_{a,b,k} (x) & := \begin{cases}
ax^{2^k}+b & \text{if $x \in \F_{q}$,}\\
\infty & \text{if $x = \infty$.}
\end{cases}
\\[2ex]
\psi_{a,b,k} (x) & := \begin{cases}
\infty & \text{if $a x^{2^k} + b = 0$,}\\
0 & \text{if $x = \infty$,}\\
\dfrac{1}{a x^{2^k}+b} & \text{otherwise.}\\
\end{cases}
\end{align*}

From Lemma \ref{ros_lemma} we get the following result.

\begin{lemma}\label{1to1_lemma}
The maps $\theta_{a,b,k}$ and $\psi_{a,b,k}$ are one-to-one.
\end{lemma}
\begin{proof}
First we notice that $\theta_{a,b,k} (x) = \infty$ only if $x = \infty$. If $\theta_{a,b,k} (x_1) = \theta_{a,b,k} (x_2)$ for some $x_1, x_2 \in \Fq$ then $a x_1^{2^k} = a x_2^{2^k}$, namely $x_1^{2^k} = x_2^{2^k}$ because $a \neq 0$. Let $\alpha = x_1^{2^k} = x_2^{2^k}$. If $\alpha = 0$, then $x_1 = x_2 = 0$. If $\alpha \neq 0$, then $x_1 = x_2$ according to Lemma \ref{ros_lemma} since $\gcd(2^k, 2^n-1) = 1$. 

As regards the maps $\psi_{a,b,k}$, we have that $\psi_{a,b,k} (x) = \infty$ only if $x$ is a solution of the equation $a x^{2^k} + b = 0$. Since $\theta_{a,b,k}$ is one-to-one, this latter equation has only one solution. Moreover $\psi_{a,b,k} (x) = 0$ only for $x = \infty$. Finally, if $\frac{1}{a x_1^{2^k}+b} = \frac{1}{a x_2^{2^k}+b}$ then $\theta_{a,b,k} (x_1) = \theta_{a,b,k} (x_2)$. Hence $x_1 = x_2$. 
\end{proof}

\section{Dynamics of the map $x \mapsto ax^4+b$}\label{sec_ell}
Let $q := 2^n$ for some $n \in \N^*$ and $a, b \in \F_{q}$ with $a \neq 0$. In the current section we investigate the map $\theta_{a,b,2}$, which we simply denote by $\theta$. The map $\theta$ is defined as follows over $\Pro (\Fq)$:
\begin{align*}
\theta (x) := \begin{cases}
ax^4+b & \text{if $x \in \F_q$,}\\
\infty & \text{if $x = \infty$.}
\end{cases}
\end{align*}

The image of the map $\theta$ gives the $x$-component of the duplication map defined on a supersingular elliptic curve over a binary field (see for example \cite[Chapter 6]{kob}). More precisely, consider the elliptic curve
\[
E : y^2 + a_1 y = x^3 + a_2 x
\] 
where $a_1, a_2 \in \F_{q}$. If $P := (x_0,y_0) \in E(\Fq)$, then $2 P = (x_1, y_1)$, where
\[
x_1 = \dfrac{x_0^4+a_2^2}{a_1^2}.
\]

Hence
\[
\theta(x_0) = a x_0^4+b =  \dfrac{x_0^4+a_2^2}{a_1^2}, 
\]
with $a_1^2 := a^{-1}$ and $a_2^2 := a_1^{2} b$.

We can describe the iterations of the map $\theta$ on $\Pro(\F_{q})$ relying upon the group of rational points $E(\F_{q})$ of $E$ over $\F_{q}$ or the group of rational points $E(\F_{q^{2}})$ of $E$ over $\F_{q^{2}}$. In fact, for any $x_0 \in \Fq$ there exists some $y_0 \in \Fqt$ such that $P := (x_0, y_0) \in E(\Fqt)$. Hence $2^i P \in E(\Fqt)$ for all $i \in \N^*$. Moreover, if $y_0 \in \Fq$ then $2^i P \in E(\Fq)$ for all $i\in \N^*$.

As recalled in \cite[Remark 2.5]{wi}, which indeed refers to \cite{sil}, if $k \in \N^*$ and $E$ is an elliptic curve defined over $\F_{2^k}$, then there exist two positive integers $n_1, n_2$ such that
\[
E(\F_{2^k}) \cong \Z / n_1 \Z \times \Z / n_2 \Z
\]
with $n_1 \mid \gcd(n_2, 2^k-1)$. Moreover, as recalled in \cite[Section 4]{wi}, the cardinality $|E(\F_{2^k})|$ is odd.
Since the map $\theta$ is related to the duplication map over $E$, we are interested in studying the action of the pair of classes $([2], [2])$ on the elements belonging to $\Z / n_1 \Z \times \Z / n_2 \Z$.

\subsection{Length of the cycles}\label{sec_length}
Let $E(\Fqt)$ be isomorphic to
\[
A := \Z / n_1 \Z \times \Z / n_2 \Z
\]
for some $n_1, n_2 \in \N^*$. 
If $x_0 \in \Fq$, then $x_0$ is the $x$-coordinate of a point $P := ([a_1],[a_2])$ in $A$. Moreover $P$ and $-P$ are the only points whose $x$-coordinate is $x_0$. 

For $i \in \{1, 2 \}$ we define 
\begin{align*}
d_i & := \gcd(a_i, n_i);\\
m_i & := \min \{k \in \N^*: 2^{k} a_i \approx a_i \bmod{n_i} \}.
\end{align*}

Moreover we define
\begin{align*}
m & := \min \{k \in \N^*: 2^{k} P \approx P \};\\
l & := \lcm(m_1, m_2).
\end{align*}

In the next theorem we show what the length of the orbit $O(x_0)$ is. 

\begin{theorem}
The following hold.
\begin{enumerate}[leftmargin=*]
\item If $i \in \{1 ,2 \}$, then $m_i$ is the smallest of the positive integers $k$ such that $2^k \approx 1 \bmod \frac{n_i}{d_i}$.
\item The length of $O(x_0)$ in $G_q (\theta)$ is $m$. Moreover we have that
\[
m = \begin{cases}
l & \text{if $2^l P \approx P$,}\\
2 l & \text{otherwise}.
\end{cases}
\] 
\end{enumerate}

\end{theorem}
\begin{proof}
We prove separately the statements.
\begin{enumerate}[leftmargin=*]
\item We have that $2^k a_i \approx  a_i \bmod{n_i}$ for some $k \in \N^*$ if and only if $n_i$ divides $(2^k + 1) a_i$ or $(2^k - 1) a_i$. According to a well-known arithmetical lemma, this latter is equivalent to the fact that $\frac{n_i}{d_i}$ divides $(2^k + 1)$ or $(2^k - 1)$. Hence the claim for $m_i$ is proved. 

\item We notice that $\theta^k (x_0) = x_0$ for some positive integer $k$ if and only if $2^k P \approx P$, namely if and only if 
\[
([2^k a_1], [2^k a_2]) = ([a_1],[a_2]) \ \text{ or } \ ([2^k a_1], [2^k a_2]) = ([-a_1], [-a_2]).
\]

Therefore $\theta^m (x_0) = x_0$. Since $m$ must be a common multiple of $m_1$ and $m_2$, we have that $l := \lcm(m_1, m_2)$ divides $m$. Moreover $2^l P \approx P$ only if $[2^l a_1] = [a_1]$ and $[2^l a_2] = [-a_2]$ or viceversa. In both cases we have that $2^{2l} P = P$. \qedhere
\end{enumerate}
\end{proof} 

\begin{remark}
We have proved that the length of $O(x_0)$ with $x_0 \in \Pro (\Fq)$ depends on $\gcd(a_1,n_1)$ and $\gcd(a_2,n_2)$, where $([a_1], [a_2])$ is a point in $E(\Fqt)$ whose $x$-coordinate is $x_0$. Therefore, the possible lengths of the cycles in $G_{q} (\theta)$ can be computed relying upon the pairs $(d_1, d_2)$, where $d_1$ and $d_2$ are (positive) divisors of $n_1$ and $n_2$ respectively and
\[
E (\Fqt) \cong \Z / n_1 \Z \times \Z / n_2 \Z.
\]
\end{remark}

If we focus on rational points of $E(\Fq)$, we can state a more precise result. With a little abuse of notation we suppose that $E(\Fq) \cong A$ where 
\[
A := \Z / n_1 \Z \times \Z /  n_2 \Z 
\] 
for some positive integers $n_1$ and $n_2$.

\begin{lemma}\label{lem_ratq}
The following hold.
\begin{enumerate}[leftmargin=*]
\item If $([a_1],[a_2]) \in A$, then $\gcd(a_i, n_i) = \gcd(2 a_i,n_i)$ for $i \in \{1, 2 \}$.
\item Let $d_1$ and $d_2$ be two positive integers such that
\[
d_1 \mid n_1 \quad \text{ and } \quad d_2 \mid n_2.
\]
If $m_1 := \frac{n_1}{d_1}$ and $m_2 := \frac{n_2}{d_2}$, then there are $\phi \left( m_1 \right) \cdot \phi \left( m_2 \right)$ points $([a_1], [a_2]) \in A$ such that 
\[
\gcd(a_1,n_1) = d_1 \quad \text{ and } \quad \gcd(a_2, n_2) = d_2.
\]
\end{enumerate}
\end{lemma}
\begin{proof}
Let $i \in \{1, 2 \}$. We prove separately the statements.
\begin{enumerate}[leftmargin=*]
\item The assertion follows from the fact that an integer $m$ divides $a_i$ and $n_i$ if and only if $m$ divides $2 a_i$ and $n_i$ because $n_i$ is odd.
\item We consider the ring isomorphism
\begin{align*}
\psi_i: d_i \Z / n_i \Z & \to  \Z / m_i \Z \\
d_i k + n_i \Z & \mapsto k + m_i \Z.
\end{align*}

Let $[a_i] \in \Z / n_i \Z$. Then $\gcd(a_i, n_i) = d_i$ if and only if $a_ i = d_i k$ for some integer $k$ coprime to $n_i$. This is equivalent to saying that $[k]$ is a unit in $\Z / m_i \Z$. Hence there are $\phi(m_i)$ classes $[a_i]$ in $\Z / n_i \Z$ such that $\gcd(a_i, n_i) = d_i$. \qedhere 
\end{enumerate}
\end{proof}

\begin{example}
Let $\Fq$ be the field with $q := 2^5$ elements. We denote by $g$ a generator of $\F^*_{q}$. We construct the graph $G_{q} (\theta)$, where $\theta$ is defined on $\Pro(\F_{q})$ as follows: 
\[
\theta (x) := \begin{cases}
g x^4+g^3 & \text{if $x \in \F_{q}$,}\\
\infty & \text{if $x = \infty$.}
\end{cases}
\]

The map $\theta$ is related to the duplication map defined over the curve 
\[
E : y^2 + g^{15} y = x^3 + gx.
\] 

The vertices of $G_q (\theta)$ are labelled by  the exponents of $g^i$ for $i \in [0,30]$. The remaining two vertices are labelled by $\infty$ and `0' (the zero of $\Fq$).
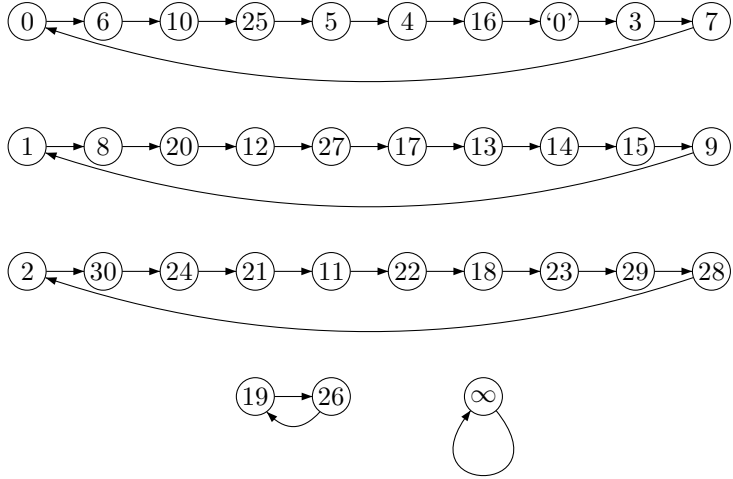
\begin{center}
    \begin{picture}(95, 15)(0,-5)
    \gasset{Nw=5,Nh=5,Nmr=2.5,curvedepth=0}
    \thinlines
    \footnotesize
    \node(N1)(0,0){$0$}
    \node(N2)(10,0){$6$}
    \node(N3)(20,0){$10$}
    \node(N4)(30,0){$25$}
    \node(N5)(40,0){$5$}
    \node(N6)(50,0){$4$}
    \node(N7)(60,0){$16$}
    \node(N8)(70,0){`0'}
    \node(N9)(80,0){$3$}
    \node(N10)(90,0){$7$}

    \drawedge(N1,N2){}
    \drawedge(N2,N3){}
    \drawedge(N3,N4){}
    \drawedge(N4,N5){}
    \drawedge(N5,N6){}
    \drawedge(N6,N7){}
    \drawedge(N7,N8){}
    \drawedge(N8,N9){}
    \drawedge(N9,N10){}
    
    \gasset{curvedepth=8}
    \drawedge(N10,N1){}
\end{picture}
\end{center}

\begin{center}
\begin{picture}(95, 15)(0,-5)
    \gasset{Nw=5,Nh=5,Nmr=2.5,curvedepth=0}
    \thinlines
    \footnotesize
    \node(N1)(0,0){$1$}
    \node(N2)(10,0){$8$}
    \node(N3)(20,0){$20$}
    \node(N4)(30,0){$12$}
    \node(N5)(40,0){$27$}
    \node(N6)(50,0){$17$}
    \node(N7)(60,0){$13$}
    \node(N8)(70,0){$14$}
    \node(N9)(80,0){$15$}
    \node(N10)(90,0){$9$}

    \drawedge(N1,N2){}
    \drawedge(N2,N3){}
    \drawedge(N3,N4){}
    \drawedge(N4,N5){}
    \drawedge(N5,N6){}
    \drawedge(N6,N7){}
    \drawedge(N7,N8){}
    \drawedge(N8,N9){}
    \drawedge(N9,N10){}
    
    \gasset{curvedepth=8}
    \drawedge(N10,N1){}
\end{picture}
\end{center}

\begin{center}
\begin{picture}(95, 20)(0,-10)
    \gasset{Nw=5,Nh=5,Nmr=2.5,curvedepth=0}
    \thinlines
    \footnotesize
    \node(N1)(0,0){$2$}
    \node(N2)(10,0){$30$}
    \node(N3)(20,0){$24$}
    \node(N4)(30,0){$21$}
    \node(N5)(40,0){$11$}
    \node(N6)(50,0){$22$}
    \node(N7)(60,0){$18$}
    \node(N8)(70,0){$23$}
    \node(N9)(80,0){$29$}
    \node(N10)(90,0){$28$}

    \drawedge(N1,N2){}
    \drawedge(N2,N3){}
    \drawedge(N3,N4){}
    \drawedge(N4,N5){}
    \drawedge(N5,N6){}
    \drawedge(N6,N7){}
    \drawedge(N7,N8){}
    \drawedge(N8,N9){}
    \drawedge(N9,N10){}
    
    \gasset{curvedepth=8}
    \drawedge(N10,N1){}
\end{picture}
\end{center}

\begin{center}
\begin{picture}(55, 15)(-10,-10)
    \gasset{Nw=5,Nh=5,Nmr=2.5,curvedepth=0}
    \thinlines
    \footnotesize
    \node(N1)(0,0){$19$}
    \node(N2)(10,0){$26$}
    \node(N3)(30,0){$\infty$}

    \drawedge(N1,N2){}

    \gasset{curvedepth=4}
    \drawedge(N2,N1){}
        
    \drawloop[loopangle=-90](N3){}
\end{picture}
\end{center}
\end{example}

Using Sage \cite{sage} we get that $E (\Fq) \cong \Z / 41 \Z$. Adopting the notations of the current section and Lemma \ref{lem_ratq} we have that $n_1= 1$ and $n_2 = 41$. 
The set of (positive) divisors of $n_2$ is
\[
D = \{1, 41 \}.
\] 

Let $d_2 := 1$ and $m_2 := 41$. Then $10$ is the smallest $k \in \N^*$ such that 
\[
2^k \approx 1 \bmod 41.
\]

Therefore there are $\phi(m_2) = 40$ points in $E(\Fq)$ whose $x$-coordinates belong to cycles of length $10$. Since exactly two points have the same $x$-coordinate, these $40$ points form $2$ cycles of length $10$. We notice in passing that the two cycles are the orbits $O(g^0)$ and $O(g^2)$. 

Let $d_2 := 41$ and $m_2 := 1$. This case corresponds to the element $\infty$, which forms a cycle of length $1$.

Now we consider the group $E(\Fqt)$. We have that $E(\Fqt) \cong \Z / 1025 \Z$. Since $1025 = 5^2 \cdot 41$, the set of divisors of $1025$ is
\[
D = \{ 1, 5, 25, 41, 205, 1025 \}.
\] 

We notice in particular that for $d_2 := 205$ we have $m_2 := \frac{1025}{205} = 5$. Then $2$ is the smallest $k \in \N^*$ such that 
\[
2^k \approx 1 \bmod 5.
\]

Hence there are two rational points in $E(\Fqt)$ whose $x$-coordinates form the cycle of length $2$ in $G_q (\theta)$. 
We get the remaining cycle of length $10$ setting $d_2 := 25$ and correspondingly $m_2 := \frac{1025}{25} = 41$.

\section{Dynamics of the maps $a x^{2^k} +b$}\label{sec_x2k}
Let $n \in \N^*$. In this section we study the maps $\theta_{a,b,k}$, where $a, b \in \F_{2^n}$ with $a \neq 0$ and $k \in \N$, defined on $\Pro(\F_{2^n})$ as in  Section \ref{sec_intro}.
We adopt the shorter notation $\theta_{a,b}$ for the map $\theta_{a,b,2}$. In Lemmas \ref{rel_1} and \ref{rel_2} we show how the dynamics of the map $\theta_{a,b,k}$ can be related to the dynamics of the map $\theta_{c,d}$ for suitable coefficients $c$ and $d$ in $\cFn$.
For any $q \in \N^*$ and any $t \in \N$ we define 
\[
s^q_t := \begin{cases}
0 & \text{if $t =0$,}\\
\sum_{i=0}^{t-1} q^i & \text{if $t >0$.}
\end{cases}
\]

\begin{lemma}\label{rel_1}
Let $q := 2^k$ for some $k \in \N$. Let $m \in \N^*$ and $a, b \in \Fn$ with $a \neq 0$. If $\psi(x) := a x^q +b$ for any $x \in \F_{2^n}$, then
\[
\psi^m (x) = a^{s_{m}} x^{q^m} + b_m,
\]
where $s_t := s_t^q$ for any $t \in [0, m-1]$ and
\[
b_m := \sum_{t=0}^{m-1} a^{s_t} b^{q^t}.
\]
\end{lemma}
\begin{proof}
We prove the statement by induction on $m \in \N^*$.

If $m = 1$, then
\[
\psi^{1} (x)  = \psi(x) = a x^q + b = a^{s_1} x^{q^1} + b_1 .
\]

Now we prove the inductive step. If $m \geq 1$, then 
\begin{align*}
\psi^{m+1} (x) & = \psi (\psi^m (x)) = a (a^{s_{m}} x^{q^m} + b_m)^q + b \\
& = a^{1 + q s_m} x^{q^{m+1}} + a b_m^q + b. 
\end{align*}

We notice that
\begin{align*}
1 + q s_m & = 1 + \sum_{t=0}^{m-1} q^{t+1} = 1 + \sum_{t=1}^{m} q^t = \sum_{t=0}^{m} q^t = s_{m+1};\\
a b_m^q + b & = a \sum_{t=0}^{m-1} a^{q s_t} b^{q^{t+1}} + b = \sum_{t=0}^{m-1} a^{q s_t+1} b^{q^{t+1}} + b \\
& = \sum_{t=0}^{m-1} a^{q s_t+1} b^{q^{t+1}}  + b = \sum_{t=0}^{m-1} a^{s_{t+1}} b^{q^{t+1}} + b \\
& = \sum_{t=1}^m a^{s_{t}} b^{q^{t}} + b = b_{m+1}.
\end{align*}

Therefore $ \psi^{m+1} (x) = a^{s_{m+1}} x^{q^{m+1}} + b_{m+1}$. 
\end{proof}

\begin{lemma}\label{rel_2}
Let $k$ be an integer  with $k \geq 2$ and $a, b \in \Fn$ with $a \neq 0$. Then there exist some $c, d \in \cFn$ such that for any $x \in \Pro(\Fn)$ we have that
\begin{align*}
\theta^{\frac{k}{2}}_{c,d} (x) & = \theta_{a,b,k} (x) \quad  \text{if $k$ is even,}\\
\theta^{k}_{c,d} (x) & = \theta_{a,b,k}^2 (x) \quad  \text{if $k$ is odd.}
\end{align*}
\end{lemma}
\begin{proof}
If $c, d \in \cFn$ and $j \in \N^*$, then from Lemma \ref{rel_1} we get that
\[
\theta^j_{c,d} (x) = c^{s_j} x^{4^j} + d_j
\]
with $s_j := s_j^4$ and $d_j := \sum_{i=0}^{j-1} c^{s_i} d^{4^i}$.

For any $k \in \N^*$ we have that
\begin{align*}
\theta_{a,b,k} (x) & = a x^{2^k} + b, \\
\theta_{a,b,k}^2 (x) & = a^{2^k+1} x^{2^{2k}} + a b^{2^k} + b.
\end{align*}

If $k$ is even, then we set $j := \frac{k}{2}$ and we find $c, d \in \cFn$ solving the equations
\[
c^{s_j} = a \quad \text{ and } \quad \sum_{i=0}^{j-1} c^{s_i} d^{4^i} = b. 
\] 

If $k$ is odd, then we set $j := k$ and we find $c, d \in \cFn$ solving the equations
\[
c^{s_j} = a^{2^k+1} \quad \text{ and } \quad \sum_{i=0}^{j-1} c^{s_i} d^{4^i} = a b^{2^k} + b. \qedhere
\] 
\end{proof}

\begin{remark}
One of the hypotheses of Lemma \ref{rel_2} is that $k \geq 2$. Indeed, if we consider $n \geq 2$, we have that 
\begin{align*}
\theta_{a,b,0} (x) & = \theta_{a,b,n} (x),\\
\theta_{a,b,1} (x) & = \theta_{a,b,n+1} (x),
\end{align*}
because $x^{2^n} = x$ and $x^{2^{n+1}} = (x^{2^n})^2 = x ^2$.
\end{remark}

Under the same hypotheses (and using the same notations) of Lemma \ref{rel_2} we define the map $\psi$ on $\Pro(\Fn)$ as follows:
\[
\psi(x) :=
\begin{cases}
\theta_{a,b,k} (x) &  \text{if $k$ is even,}\\
\theta_{a,b,k}^2 (x) &  \text{if $k$ is odd.}
\end{cases}
\]

Moreover we define the integer
\[
m := \begin{cases}
\frac{k}{2} & \text{if $k$ is even,}\\
k &  \text{if $k$ is odd.}
\end{cases}
\]

The following holds.
\begin{lemma}\label{lem_123}
Let $x_0 \in \Pro(\Fn)$ and
\begin{align*}
l_1 & := |O(\theta_{c,d}, x_0)|, \\
l_2 & := |O(\theta_{a,b,k}, x_0)|,\\
l_3 & := |O(\psi, x_0)|.
\end{align*}

Then $l_3 = \frac{l_{1}}{\gcd(l_{1},m)}$. Moreover $l_{2} = l_{3}$ if $k$ is even, while $l_{2} \in \{l_{3}, 2 l_{3} \}$ if $k$ is odd.
\end{lemma}
\begin{proof}
From Lemma \ref{lem_orbits} we get immediately that $l_{3} = \frac{l_1}{\gcd(l_{1},m)}$ because $\psi = \theta_{c,d}^m$. If $k$ is even, then $l_{2} = l_{3}$ because $\theta_{a,b,k} = \psi$. If $k$ is odd, then from Lemma \ref{lem_orbits} we get that
\[
l_{2} = l_{3} \cdot \gcd(l_{2}, 2).
\]
Hence $l_{2} = l_{3}$ if $l_{2}$ is odd, while $l_{2} = 2 l_{3}$ if $l_{2}$ is even.
\end{proof}

\begin{example}
Let $g$ be a generator of $\F_{q}^*$, where $q:=2^5$.
In this example we consider the map $\sigma := \theta_{g^7, g^3, 3}$, which is defined for all $x \in \F_{q}$ as 
\[
\sigma(x) = g^7 x^8 + g^3.
\]

Since $k=3$, we define $\psi(x) := \sigma^2 (x)$ for all $x \in \Pro (\F_{q})$. We want to find $c, d \in \overline{\F_{q}}$ such that
\[
\theta^3_{c,d} (x) = \psi(x)
\]
for all $x \in \Pro (\F_{q})$. More explicitly, we have to solve the equations
\begin{align*}
c^{s_3} & = (g^7)^9,\\
\sum_{i=0}^{2} c^{s_i} d^{4^i} & = g^7 (g^3)^8+g^3,
\end{align*}
namely
\begin{align*}
c^{21} & = g^{63},\\
d + c d^4 + c^5 d^{16} & = g^{31}+g^3.
\end{align*}

The first equation is satisfied for $c = g^3$, while one solution of the second equation is $d = g^{15}$. We define the map
\[
\theta(x) := \theta_{c,d} (x) =  \begin{cases}
g^3 x^4+g^{15} & \text{if $x \in \F_{q}$,}\\
\infty & \text{if $x = \infty$.}
\end{cases}
\]

The map $\theta$ is related to the duplication map defined over the curve 
\[
E : y^2 + g^{14} y = x^3 + g^6 x.
\] 

We consider the group 
\[ 
E(\Fq) \cong \Z / 33 \Z.
\] 

Using the notations of Section \ref{sec_length} we have that $n_2 = 33$. The set of positive divisors of $33$ is
\[
D := \{ 1, 3, 11, 33 \}.
\]

For each $d \in D$ we can find the smallest $k \in \N^*$ such that $2^k \approx 1 \bmod d$. Such a $k$ corresponds to a possible length of an orbit in $G_q (\theta)$. We collect the values of $k$ in the set  
\[
K := \{1, 5 \}.
\]

According to the notations of Lemma \ref{lem_123} we have that $l_1 \in K$ and $l_3 = l_1$. Since $k$ is odd, the possible length of an orbit $O (\sigma, x_0)$, where $x_0$ is the $x$-coordinate of a point in $E(\Fq)$, is $1, 2, 5$ or $10$. 

We can extend our investigation to the group 
\[
E(\Fqt) \cong \Z / 33 \Z \times \Z / 33 \Z.
\]

We have $n_1 = n_2 = 33$. Moreover, the pairs of divisors of $(n_1, n_2)$ are $(d_1, d_2)$ with $d_1, d_2 \in D$. For any such a pair we can compute the smallest $k \in \N^*$ such that 
\begin{align*}
2^k \equiv 1 \bmod d_1 \ & \text{ and  } \ 2^k \equiv 1 \bmod d_2\\
& \ \text{ or } \\
2^k \equiv -1 \bmod d_1 \ & \text{ and  } \ 2^k \equiv -1 \bmod d_2
\end{align*}

The values of $k$ we get are the elements of 
\[
K := \{1, 2, 5,10 \}.
\]

Correspondingly, the possible length of an orbit $O(\sigma, x_0)$ with $x_0 \in \Pro (\Fq)$ is $1, 2, 4, 5, 10$ or $20$.

Below we represent the graph $G_q (\sigma)$. 
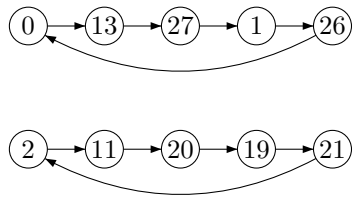
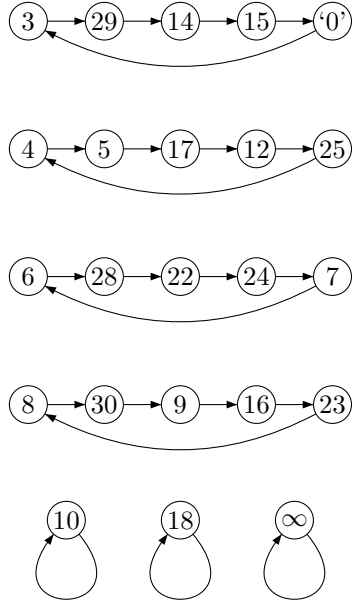
\begin{center}
\begin{picture}(50, 15)(0,-5)
    \gasset{Nw=5,Nh=5,Nmr=2.5,curvedepth=0}
    \thinlines
    \footnotesize
    \node(N1)(0,0){$0$}
    \node(N2)(10,0){$13$}
    \node(N3)(20,0){$27$}
    \node(N4)(30,0){$1$}
    \node(N5)(40,0){$26$}

    \drawedge(N1,N2){}
    \drawedge(N2,N3){}
    \drawedge(N3,N4){}
    \drawedge(N4,N5){}
    
    \gasset{curvedepth=6}
    \drawedge(N5,N1){}
\end{picture}
\end{center}

\begin{center}
\begin{picture}(50, 15)(0,-5)
    \gasset{Nw=5,Nh=5,Nmr=2.5,curvedepth=0}
    \thinlines
    \footnotesize
    \node(N1)(0,0){$2$}
    \node(N2)(10,0){$11$}
    \node(N3)(20,0){$20$}
    \node(N4)(30,0){$19$}
    \node(N5)(40,0){$21$}

    \drawedge(N1,N2){}
    \drawedge(N2,N3){}
    \drawedge(N3,N4){}
    \drawedge(N4,N5){}
    
    \gasset{curvedepth=6}
    \drawedge(N5,N1){}
\end{picture}
\end{center}

\begin{center}
\begin{picture}(50, 15)(0,-5)
    \gasset{Nw=5,Nh=5,Nmr=2.5,curvedepth=0}
    \thinlines
    \footnotesize
    \node(N1)(0,0){$3$}
    \node(N2)(10,0){$29$}
    \node(N3)(20,0){$14$}
    \node(N4)(30,0){$15$}
    \node(N5)(40,0){`0'}

    \drawedge(N1,N2){}
    \drawedge(N2,N3){}
    \drawedge(N3,N4){}
    \drawedge(N4,N5){}
    
    \gasset{curvedepth=6}
    \drawedge(N5,N1){}
\end{picture}
\end{center}

\begin{center}
\begin{picture}(50, 15)(0,-5)
    \gasset{Nw=5,Nh=5,Nmr=2.5,curvedepth=0}
    \thinlines
    \footnotesize
    \node(N1)(0,0){$4$}
    \node(N2)(10,0){$5$}
    \node(N3)(20,0){$17$}
    \node(N4)(30,0){$12$}
    \node(N5)(40,0){$25$}

    \drawedge(N1,N2){}
    \drawedge(N2,N3){}
    \drawedge(N3,N4){}
    \drawedge(N4,N5){}
    
    \gasset{curvedepth=6}
    \drawedge(N5,N1){}
\end{picture}
\end{center}

\begin{center}
\begin{picture}(50, 15)(0,-5)
    \gasset{Nw=5,Nh=5,Nmr=2.5,curvedepth=0}
    \thinlines
    \footnotesize
    \node(N1)(0,0){$6$}
    \node(N2)(10,0){$28$}
    \node(N3)(20,0){$22$}
    \node(N4)(30,0){$24$}
    \node(N5)(40,0){$7$}

    \drawedge(N1,N2){}
    \drawedge(N2,N3){}
    \drawedge(N3,N4){}
    \drawedge(N4,N5){}
    
    \gasset{curvedepth=6}
    \drawedge(N5,N1){}
\end{picture}
\end{center}

\begin{center}
\begin{picture}(50, 15)(0,-5)
    \gasset{Nw=5,Nh=5,Nmr=2.5,curvedepth=0}
    \thinlines
    \footnotesize
    \node(N1)(0,0){$8$}
    \node(N2)(10,0){$30$}
    \node(N3)(20,0){$9$}
    \node(N4)(30,0){$16$}
    \node(N5)(40,0){$23$}

    \drawedge(N1,N2){}
    \drawedge(N2,N3){}
    \drawedge(N3,N4){}
    \drawedge(N4,N5){}
    
    \gasset{curvedepth=6}
    \drawedge(N5,N1){}
\end{picture}

\begin{picture}(40, 20)(0,-10)
    \gasset{Nw=5,Nh=5,Nmr=2.5,curvedepth=0}
    \thinlines
    \footnotesize
    \node(N1)(0,0){$10$}
    \node(N2)(15,0){$18$}
    \node(N3)(30,0){$\infty$}
   
    \gasset{curvedepth=4}
    \drawloop[loopangle=-90](N1){}
    \drawloop[loopangle=-90](N2){}
    \drawloop[loopangle=-90](N3){}
\end{picture}
\end{center}

\end{example}

\section{Dynamics of the maps $(a x^{2^k}+b)^{-1}$}\label{sec_x2k1}
In this section we study the dynamics of the maps $(a x^{2^k}+b)^{-1}$, namely the maps $\psi_{a,b,k}$ defined in Section \ref{sec_pre}, where $a,b \in \Fn$ with $a \neq 0$ and $k \in \N^*$.

We can find an invertible map $\tau$ defined on $\Pro(\cFn)$ such that $\psi_{a,b,k}$ is conjugated to $\theta_{c,0,k}$ through $\tau$ for some $c \in \cFn$, namely 
\[
(\tau^{-1} \circ \psi_{a,b,k} \circ \tau) (x) = \theta_{c,0,k} (x)
\]
or equivalently
\[
(\psi_{a,b,k} \circ \tau) (x) = (\tau \circ \theta_{c,0,k}) (x) 
\]
for all $x \in \Pro(\cFn)$.

We construct the map $\tau$ as follows, for some $c_1, c_2, c_3 \in \cFn$ with $c_3 + c_1 c_2 \neq 0$:
\[
\tau(x) := \begin{cases}
\dfrac{1}{c_2} & \text{if $x = \infty$,} \\[1ex]
\infty & \text{if $c_2 x + c_3 = 0$,}\\
0 & \text{if $x = c_1$,}\\
\dfrac{x+c_1}{c_2 x+c_3} & \text{otherwise.}\\
\end{cases}
\]

Now we find the values of $c_1, c_2, c_3$ and $c$.
Let $x \in \cFn$ be an element such that $x \neq c_1$ and $c_2 x + c_3 \neq 0$. If $q :=2^k$, then 
\[
(\psi_{a,b,k} \circ \tau) (x) = (\tau \circ \theta_{c,0,k}) (x) 
\]
if and only if
\[
\dfrac{1}{a \left( \dfrac{x+c_1}{c_2 x+c_3}  \right)^{q}+b} = \dfrac{c x^{q} + c_1}{c_2 c x^{q} + c_3}.
\]
Unravelling the left-hand-side of this latter equation we get 
\[
\dfrac{c_2^{q} x^{q} +  c_3^{q}}{(a+b c_2^{q})x^{q}+ (a c_1^{q} + b c_3^{q})} = \dfrac{c x^q + c_1}{c_2 c x^{q} + c_3}.
\]

The two rational functions above are equal if we find $c, c_1, c_2, c_3$ in $\cFn$ such that
\begin{equation}\label{eq_c123}
\begin{cases}
c = c_2^q\\
c_1 = c_3^{q} \\
c_2 c = a + b c_2^q \\
c_3 = a c_1^q + b c_3^{q}
\end{cases}
\end{equation}

If we plug $c := c_2^q$ in $c_2 c = a + b c_2^q$ we get the equation
\[
c_2^{q+1} + b c_2^q + a = 0
\]
which is satisfied for some $c_2 \in \cFn$ with $c_2 \neq 0$.
We notice in passing that $c_3 + c_1 c_2 = 0$ if and only if $c_3 + c_2 c_3^q  = 0$,
namely if and only if $c_3$ is a root of  the polynomial
\[
u (x) := x + c_2 x^q \in \cFn [x].
\]

If we plug $c_1 := c_3^q$ in  $c_3  = a c_1^q + b c_3^q$ we get the equation
\[
a c_3^{q^2} + b c_3^q + c_3 = 0, 
\]
which is satisfied if and only if $c_3$ is a root of the polynomial
\[
v(x) := a x^{q^2} + b x^q + x \in \F_{2^n} [x].
\]

Since $u(x)$ and $v(x)$ are $q$-polynomial and $u'(x) = v'(x) \neq 0 $,  they have respectively $q$ and $q^2$ distinct roots in $\cFn$ (see \cite[Theorem 3.50]{lid}). Therefore
\[
c_3 + c_1 c_2 	\neq 0 
\]
for some $c_3 \in \cFn$. We also notice that, if the equations (\ref{eq_c123}) are satisfied, then $c_2, c_3 \neq 0$.

We should also check that 
\[
\psi_{a,b,k} (\tau(x)) = \tau (\theta_{c,0,k}(x)) 
\]
for $x \in \{c_1, \frac{c_3}{c_2} \}$. We discuss separately some cases.

\begin{itemize}[leftmargin=*]
\item If $b \neq 0$ and $x = c_1$, then 
\begin{align*}
\psi_{a,b,k} (\tau(x)) & = \psi_{a,b,k} (0) = \frac{1}{b}, \\
\tau (\theta_{c,0,k}(x)) & = \tau (c c_1^q) = \frac{c c_1^q +c_1}{c_2 c c_1^q + c_3} = \frac{c_2^q c_1^q + c_3^q}{a c_1^q + b c_2^q c_1^q + a c_1^q + b c_3^q} \\
& = \frac{c_2^q c_1^q + c_3^q}{b (c_2^q c_1^q + c_3^q)} = \frac{1}{b}.
\end{align*}
\item If $b = 0$ and $x = c_1$, then $c_2 = \frac{a}{c}$, $c_3 = a c_1^q$ and $\frac{c_3}{c_2} = c c_1^q$. Therefore
\begin{align*}
\psi_{a,b,k} (\tau(x)) & = \psi_{a,b,k} (0) = \infty, \\
\tau (\theta_{c,0,k}(x)) & = \tau (c c_1^q) = \infty.
\end{align*}
\item If $x = \frac{c_3}{c_2}$, then 
\begin{align*}
\psi_{a,b,k} (\tau(x)) & = \psi_{a,b,k} (\infty) = 0, \\
\tau (\theta_{c,0,k}(x)) & = \tau (c_1) = 0.
\end{align*}
\end{itemize}

\subsection{Fixed points}
If $m \in \N^*$ and $c \in \F^*_{2^m}$, then the map $\theta : = \theta_{c,0,k}$ has at least two fixed points in $\Pro (\Fm)$ since
\[
\theta (0) = 0 \quad \text{ and } \quad \theta (\infty) = \infty.
\]

For a map $\psi := \psi_{a,b,k}$, defined over $\Pro(\Fm)$ and conjugated to $\theta$ through a suitable map $\tau$ with $c_1, c_2, c_3 \in \Fm$, we can prove some results on the number of fixed points or equivalently on the number of cycles of length $1$ in $G_{2^m} (\psi)$.

\begin{theorem}\label{thm_fix}
Let  $d:= \gcd(2^k-1, 2^m-1)$.
\begin{enumerate}[leftmargin=*]
\item The elements $\frac{c_1}{c_3}$ and $\frac{1}{c_2}$ are fixed points for the map $\psi$.
\item If $(c^{-1})^{\frac{2^m-1}{d}} \neq 1$, then $\psi$ has exactly two fixed points in $\Pro(\Fm)$.
\item If $(c^{-1})^{\frac{2^m-1}{d}} = 1$, then $\psi$ has $d+2$ fixed points in $\Pro(\Fm)$. 
\end{enumerate}
\end{theorem}
\begin{proof}
We notice that the map $\tau$ establishes a one-to-one correspondence between the fixed points of $\psi$ and $\theta$.
Now we prove separately the statements.
\begin{enumerate}[leftmargin=*]
\item Since $\tau(0) = \frac{c_1}{c_3}$ and $\tau (\infty) = \frac{1}{c_2}$, we have that $\frac{c_1}{c_3}$ and $\frac{1}{c_2}$ are fixed points for $\psi_{a,b,k}$. Moreover $\frac{c_1}{c_3} \neq \frac{1}{c_2}$ because $c_1 c_2 + c_3 \neq 0$.
\item  If $x \in \F^*_{2^m}$, then 
\[
c x^{2^k} = x \Leftrightarrow x^{2^k-1} = c^{-1}.
\]
According to Proposition \ref{ros_lemma}, this latter equation has no solution in $\F_{2^m}$. Hence the only fixed points of $\psi$ are $\frac{c_1}{c_3}$ and $\frac{1}{c_2}$.
\item Similarly to the previous item, in this case the equation $x^{2^k-1} = c^{-1}$ has $d$ distinct solutions in $\F_{2^m}$. Since $0$ and $\infty$ are fixed points of $\theta$, we get the result. \qedhere
\end{enumerate}
\end{proof}

\begin{example}
Let $g$ a generator of $\F^*_{2^5}$. We consider the map $\psi := \psi_{g,g^2,2}$, whose image for all but two points of $\Pro(\F_{2^5})$ is 
\[
\psi(x) = \dfrac{1}{g x^4 + g^2}.
\]

Using GAP \cite{gap} we find the following values for $c_1$, $c_2$, $c_3$ and $c$:
\[
c_1 := g; \quad c_2 := g^3; \quad c_3 := g^8; \quad c := g^{12}.
\]

We notice that $m = 5$ and $k = 2$ according to the notations of Theorem \ref{thm_fix}. Since $2^k-1$ and $2^m-1$ are coprime and $(c^{-1})^{2^m-1} = 1$, the map $\psi$ has $3$ fixed points in $\Pro (\Fm)$.

The graph $G_{2^5} (\psi)$ is represented in the following figure.
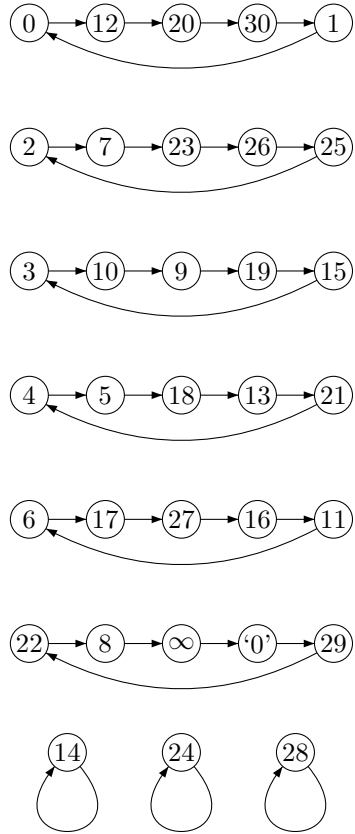
\begin{center}
\begin{picture}(50, 15)(0,-5)
    \gasset{Nw=5,Nh=5,Nmr=2.5,curvedepth=0}
    \thinlines
    \footnotesize
    \node(N1)(0,0){$0$}
    \node(N2)(10,0){$12$}
    \node(N3)(20,0){$20$}
    \node(N4)(30,0){$30$}
    \node(N5)(40,0){$1$}

    \drawedge(N1,N2){}
    \drawedge(N2,N3){}
    \drawedge(N3,N4){}
    \drawedge(N4,N5){}
    
    \gasset{curvedepth=6}
    \drawedge(N5,N1){}
\end{picture}
\end{center}

\begin{center}
\begin{picture}(50, 15)(0,-5)
    \gasset{Nw=5,Nh=5,Nmr=2.5,curvedepth=0}
    \thinlines
    \footnotesize
    \node(N1)(0,0){$2$}
    \node(N2)(10,0){$7$}
    \node(N3)(20,0){$23$}
    \node(N4)(30,0){$26$}
    \node(N5)(40,0){$25$}

    \drawedge(N1,N2){}
    \drawedge(N2,N3){}
    \drawedge(N3,N4){}
    \drawedge(N4,N5){}
    
    \gasset{curvedepth=6}
    \drawedge(N5,N1){}
\end{picture}
\end{center}

\begin{center}
\begin{picture}(50, 15)(0,-5)
    \gasset{Nw=5,Nh=5,Nmr=2.5,curvedepth=0}
    \thinlines
    \footnotesize
    \node(N1)(0,0){$3$}
    \node(N2)(10,0){$10$}
    \node(N3)(20,0){$9$}
    \node(N4)(30,0){$19$}
    \node(N5)(40,0){$15$}

    \drawedge(N1,N2){}
    \drawedge(N2,N3){}
    \drawedge(N3,N4){}
    \drawedge(N4,N5){}
    
    \gasset{curvedepth=6}
    \drawedge(N5,N1){}
\end{picture}
\end{center}

\begin{center}
\begin{picture}(50, 15)(0,-5)
    \gasset{Nw=5,Nh=5,Nmr=2.5,curvedepth=0}
    \thinlines
    \footnotesize
    \node(N1)(0,0){$4$}
    \node(N2)(10,0){$5$}
    \node(N3)(20,0){$18$}
    \node(N4)(30,0){$13$}
    \node(N5)(40,0){$21$}

    \drawedge(N1,N2){}
    \drawedge(N2,N3){}
    \drawedge(N3,N4){}
    \drawedge(N4,N5){}
    
    \gasset{curvedepth=6}
    \drawedge(N5,N1){}
\end{picture}
\end{center}

\begin{center}
\begin{picture}(50, 15)(0,-5)
    \gasset{Nw=5,Nh=5,Nmr=2.5,curvedepth=0}
    \thinlines
    \footnotesize
    \node(N1)(0,0){$6$}
    \node(N2)(10,0){$17$}
    \node(N3)(20,0){$27$}
    \node(N4)(30,0){$16$}
    \node(N5)(40,0){$11$}

    \drawedge(N1,N2){}
    \drawedge(N2,N3){}
    \drawedge(N3,N4){}
    \drawedge(N4,N5){}
    
    \gasset{curvedepth=6}
    \drawedge(N5,N1){}
\end{picture}
\end{center}

\begin{center}
\begin{picture}(50, 15)(0,-5)
    \gasset{Nw=5,Nh=5,Nmr=2.5,curvedepth=0}
    \thinlines
    \footnotesize
    \node(N1)(0,0){$22$}
    \node(N2)(10,0){$8$}
    \node(N3)(20,0){$\infty$}
    \node(N4)(30,0){`0'}
    \node(N5)(40,0){$29$}

    \drawedge(N1,N2){}
    \drawedge(N2,N3){}
    \drawedge(N3,N4){}
    \drawedge(N4,N5){}
    
    \gasset{curvedepth=6}
    \drawedge(N5,N1){}
\end{picture}
\end{center}

\begin{center}
\begin{picture}(50, 18)(-5,-10)
    \gasset{Nw=5,Nh=5,Nmr=2.5,curvedepth=0}
    \thinlines
    \footnotesize
    \node(N1)(0,0){$14$}
    \node(N2)(15,0){$24$}
    \node(N3)(30,0){$28$}
   
    \gasset{curvedepth=4}
    \drawloop[loopangle=-90](N1){}
    \drawloop[loopangle=-90](N2){}
    \drawloop[loopangle=-90](N3){}
\end{picture}
\end{center}
\end{example}

The map $\psi$ is conjugated to the map $\theta := \theta_{c,0,2}$ defined over $\Pro(\F_{2^5})$. The map $\theta$ describes the $x$-coordinate of the duplication map defined over the elliptic curve
\[
E : y^2 + g^{25} y = x^3.
\]

Since $E(\F_{2^5}) \cong \Z / 33 \Z$ and the positive divisors of $33$ are $1, 3, 11$ and $33$, the rational points of $E(\F_{2^5})$ give rise to cycles of lengths $1$ and $5$.

Moreover $E(\F_{2^{10}}) \cong \Z / 33 \Z \times \Z / 33 \Z$. We have that $n_1 = n_2 = 33$. Consider the pairs of positive divisors $(d_1, d_2)$ of $(n_1, n_2)$. We still get that the possible cycle lengths of elements which are $x$-coordinates (in $\F_{2^5}$) of rational points in $E(\F_{2^{10}})$ are $1$ and $5$.

\end{document}